\documentclass[12pt]{article}
\usepackage{latexsym}
\usepackage{epsfig}
\usepackage{theorem}
\usepackage{amsmath}
\usepackage{amssymb}
\usepackage{graphics}
\usepackage{epsfig}

\newenvironment{proof}{
\noindent {\bf Proof.\hskip 3mm}}%
{\mbox{}\hfill\rule{0.5em}{0.809em}\par}
\newtheorem{theorem}{{\bf Theorem}}[section]
\newtheorem{lemma}{{\bf Lemma}}[section]

\newtheorem{conjecture}{{\bf Conjecture }}[section]
\newtheorem{problem}[conjecture]{{\bf Problem}}

\newtheorem{fact}{{\bf Fact}}[section]
\newenvironment{pf1}{\noindent {\bf Proof of Theorem \ref{main-result}.}}{\hfill\rule{2mm}{3mm}\par\medskip}

\newcommand\path{maximal bichromatic path }

\begin{document}
\title{ An improved bound on acyclic chromatic index of planar graphs
\thanks{ This work is  supported by  research grants NSFC (11001055) and NSFFP (2010J05004, 2011J06001).
}}
\author{Yue Guan, Jianfeng Hou\thanks{The corresponding author: jfhou@fzu.edu.cn},
Yingyuan Yang\\Center for Discrete Mathematics, Fuzhou University,
\\Fujian, P. R. China, 350002}
\date{}
\maketitle
\textbf{Abstract:} A proper edge coloring of a graph $G$ is called
acyclic if there is no bichromatic cycle in $G$. The acyclic
chromatic index of $G$, denoted by $\chi'_a(G)$, is the least number
of colors $k$ such that $G$ has an acyclic edge $k$-coloring.
Basavaraju et al. [Acyclic edge-coloring of planar graphs, SIAM J.
Discrete Math. 25 (2) (2011), 463--478] showed that $\chi'_a(G)\le
\Delta(G)+12$ for planar graphs $G$ with maximum degree $\Delta(G)$. In this paper, the bound is improved to $\Delta(G)+10$.


\section{Introduction}

In this paper, all considered graphs are finite, simple and
undirected. We use $V(G)$, $E(G)$, $\delta(G)$ and $\Delta(G)$ (or
$V, E, \delta$ and $\Delta$ for simple) to denote the vertex set,
the edge set, the minimum degree and the maximum degree of a graph
$G$, respectively. For a vertex $v\in V(G)$, let $N(v)$ denote the
set of vertices adjacent to $v$ and $d(v)=|N(v)|$ denote the
$degree$ of $v$. Let $N_k(v)=\{x\in N(v)|d(x)=k\}$ and
$n_k(v)=|N_k(v)|$. A vertex of degree $k$ is called a {\em
$k$-vertex}. We write a {\em $k^+$-vertex} for a vertex of degree at
least $k$, and a {\em $k^-$-vertex} for that of degree at most $k$.
The $girth$ of a graph $G$, denoted by $g(G)$, is the length of its
shortest cycle.

As usual $[k]$ stands for the set $\{1,2,\ldots,k\}$.

A $proper$ $edge$ $k$-$coloring$ of a graph $G$ is a mapping $\phi$
from  $E(G)$ to the color set $[k]$ such that no pair of adjacent
edges are colored with the same color.  A proper  edge coloring of a
graph $G$ is called {\em acyclic} if there is no bichromatic cycle
in $G$. In other words, if the union of any two color classes
induces a subgraph of $G$ which is a forest.  The {\em acyclic
chromatic index} of $G$, denoted by $\chi'_a(G)$, is the least
number of colors $k$ such that $G$ has an acyclic edge $k$-coloring.

Acyclic edge coloring has been widely studied over the past twenty
years. The first general linear upper bound on $\chi'_a(G)$ was
found by Alon et al. \cite{Alon1991} who proved that
$\chi'_a(G)\leq{64\Delta(G)}$. This bound was improved to
$16\Delta(G)$ by Molloy and Reed \cite{molloy1998}.

In 2001, Alon, Sudakov and Zaks \cite{Alon2001} stated the Acyclic
Edge Coloring Conjecture, which says that
$\chi'_a(G)\leq{\Delta(G)+2}$ for every graph $G$. Though the best known upper bound for general case is far from the conjecture $\Delta(G)+2$, the conjecture has been shown to be true for some special classes of graphs, including graphs with maximum degree at most three \cite{and2012},  non
regular graphs with maximum degree at most four \cite{basa2009}, outerplanar graphs \cite{hou2009}, graphs with large girth \cite{Alon2001}, and so on.

Fiedorowicz et al. \cite{fie2008} gave an upper bound of
$2\Delta(G)+29$ for planar graphs and of $\Delta(G)+6$ for
triangle-free planar graphs.  Independently, Hou et al.
\cite{hou2008} proved that $\chi'_a(G)\leq
\max\{2\Delta(G)-2,\Delta(G)+22\}$ for planar graphs and
$\chi'_a(G)\leq \Delta(G)+2$ for planar graphs with girth at least
5. Borowiecki and Fiedorowicz \cite{boro2009} showed that $\chi'_a(G)\leq
\Delta(G)+15$ for  planar graphs without cycles of length $4$.  Hou et al. \cite{hou2011} improved the bound to $\Delta(G)+4$.
 Recently, Borowiecki et al. \cite{basaSIAM} showed that
$\chi_{a}'(G)\leq \Delta(G)+12$ for planar graphs. In this
paper, the bound is improved to $\Delta(G)+10$ as follows.

\begin{theorem}\label{main-result}
Let $G$ be a planar graph with maximum degree $\Delta(G)$. Then
$\chi_{a}'(G)\leq \Delta(G)+10$.
\end{theorem}

In Section 2, we give a structural lemma of planar graphs using discharging method. In Section 3, we consider the properties of so-called acyclically edge critical graphs. In Section 4, we give the proof of Theorem \ref{main-result}.

\section{A structural lemma}

In this section, we give a structural lemma of planar graphs using the well-know discharging method. This lemma plays an important role in the proof of Theorem \ref{main-result}.

\begin{lemma}\label{unavoidable configurations}
Let $G$ be a connected  planar graph. Then there exists a vertex $v$ with $k$
neighbors $v_1,v_{2},\cdots,v_{k}$ with $d(v_{1})\leq\cdots\leq
d(v_{k})$ such that at least one of the following is true:

\medskip

$(\mathcal{A}_1)$ $k\leq 2$;

\medskip

$(\mathcal{A}_2)$ $k=3$ with $d(v_{1})\leq 11$;

\medskip

$(\mathcal{A}_3)$ $k=4$ with $d(v_{1})\leq 7$ and $d(v_{2})\leq 9$;

\medskip

$(\mathcal{A}_4)$ $k=5$ with $d(v_{1})\leq 6, d(v_{2})\leq 7$ and $d(v_{3})\leq
 8$.
\end{lemma}

\begin{proof}  By contradiction, let $G$ be a planar graph
with a fixed embedding in the plane and suppose that $G$ is a
counterexample to the lemma.  Let $F(G)$ denote its face set. We may
assume that $G$ is a 2-connected triangulation, for otherwise, we
could add edges to $G$ obtaining a triangulation $G'$. If none of
$(\mathcal{A}_1)$-$(\mathcal{A}_4)$ holds for $G$, then clearly none
of $(\mathcal{A}_1)$-$(\mathcal{A}_4)$ holds for $G'$.

We associate a charge $\varphi(x)$ to each element  $x\in V(G)\cup
F(G)$ by $\varphi(x)=2d(x)-6$ if $x\in V(G)$, and
$\varphi(x)=d(x)-6$ if $x\in F(G)$.  We shall transfer the charge of
the vertices to the faces of $G$, in such a way that the total charge of
vertices and faces remains constant. The discharging rules are
defined as follows.

\medskip

($R_{1}$) Every  $6^+$-vertex $v$ of $G$ gives
$\frac{\varphi(v)}{d(v)}$ to each incident face.

\medskip

($R_{2}$) Let $v$ be a 4-vertex of $G$. If its neighbors  are
$8^{+}$-vertices, then $v$ gives $\frac{1}{2}$ to each incident
face. Otherwise, $v$ is adjacent to at most one vertex whose degree is at most seven, say $v_1$. Then $v$ gives $\frac{4}{5}$ to each each incident
face  containing $vv_{1}$ and $\frac{1}{5}$ to each the other
incident face.

\medskip

($R_{3}$) Let $v$ be a 5-vertex with neighbors $v_{1}, v_{2},\cdots,
v_{5}$ such that $d(v_{1})\leq\cdots\leq d(v_{5})$.

($R_{3.1}$) If $d(v_{1})\geq 7$, then $v$ gives $\frac{4}{5}$ to
each incident face.

($R_{3.2}$) If $d(v_{1})\leq 6$ and  $d(v_{2})\geq8$, then $v$ gives
$\frac{5}{4}$ to each incident face containing  $vv_{1}$ and
$\frac{1}{2}$ to each the other incident face.

($R_{3.3}$) If $d(v_{1})\leq 6$ and $d(v_{2})\leq7$, then
$d(v_{3})\geq9$. When the edge $vv_{1}$ and $vv_{2}$ belonging to a
same face $f$, then $v$ gives 1 to face $f$, $\frac{5}{6}$ to the
face  only containing $vv_{1}$ or $vv_{2}$, $\frac{2}{3}$ to each
remaining incident face. Otherwise, $v$ gives $\frac{13}{15}$ to
each face  containing  $vv_{1}$ or $vv_{2}$, $\frac{8}{15}$ to each
the other incident face.

By Euler's formula, we have
\begin{eqnarray}\label{type I}
 \sum_{v\in V(G)} \varphi(v)+\sum_{f\in F(G)} \varphi(f)= \sum_{v\in
 V(G)}(2d(v)-6)+\sum_{f\in F(G)}(d(f)-6)=-12.
\end{eqnarray}

After doing all possible charge transfers once, let $\varphi'(x)$ be
the resulting charge on each element $x\in V(G)\cup F(G)$. Then

\begin{eqnarray}\label{2}
 \sum_{x\in V(G)\cup F(G)} \varphi'(x) = \sum_{x\in V(G)\cup F(G)} \varphi(x)=-12.
\end{eqnarray}

Now we shall show $\varphi'(x)\geq 0$ for any $x\in
V(G)\cup F(G)$, a contradiction to (2).

Let $v$ be any vertex of $G$. Clearly, $\varphi'(v)=\varphi(v)=0$,
if $d(v)=3$, and
$\varphi'(v)=\varphi(v)-d(v)\times\frac{\varphi(v)}{d(v)}=0$, if
$d(v)\geq6$
 by $(R_1)$.

Suppose that  $v$ is a 4-vertex. Then $\varphi(v)=2$. If the
neighbors of $v$ are $8^{+}$-vertices, then
$\varphi'(v)=\varphi(v)-4\times \frac{1}{2}=0$; otherwise,
$\varphi'(v)=\varphi(v)-2\times \frac{4}{5}-2\times \frac{1}{5}=0$
by $(R_2)$.

Suppose that  $v$ is a 5-vertex. Then $\varphi(v)=4$. If
$d(v_{1})\geq7$, then
$\varphi'(v)\geq\varphi(v)-5\times\frac{4}{5}=0$ by $(R_{3.1})$. If
$d(v_{1})\leq 6$ and  $d(v_{2})\geq 8$, then
$\varphi'(v)\geq\varphi(v)-2\times\frac{5}{4}-3\times\frac{1}{2}=0$
by $(R_{3.2})$. Otherwise,  $d(v_{1})\leq 6$,  $d(v_{2})\leq7$ and
$d(v_{3})\geq9$, since $G$ does not contain the configuration $(\mathcal{A}_4)$. In this case, if $vv_{1}$ and $vv_{2}$ are incident
with  a same face, then
$\varphi'(v)\geq\varphi(v)-1-2\times\frac{5}{6}-2\times\frac{2}{3}=0$.
Otherwise,
$\varphi'(v)\geq\varphi(v)-4\times\frac{13}{15}-\frac{8}{15}=0$.

Let $f=[xyz]$ be any 3-face of $G$ with $d(x)\leq d(y)\leq d(z)$.
Then $\varphi(f)=-3$. If $d(x)\geq6$, then $f$ gets at least 1 from
each incident vertex,  $\varphi'(f)\geq\varphi(f)+3\times1=0$. If
$d(x)=6$, then $d(y)\ge 12$. This implies that $y$ (or $z$) gives at
least $\frac 3 2$ to $f$ by $(R_1)$, so
$\varphi'(f)\geq\varphi(f)+2\times\frac{3}{2}=0$.

Suppose that $d(x)=4$. Then $x$ gives at least $\frac 1 5$ to $f$.
We consider the degree of $y$. If $d(y)\ge 10$, then $\varphi'(f)\geq \varphi(f)+\frac{1}{5}+2\times\frac{7}{5}=0$.
If $8\le d(y)\le 9$, since $G$ does not contain the configuration $(\mathcal{A}_3)$, $v$ is not adjacent to a $7^-$-vertex. This implies
$x$ gives  $\frac 12$ to $f$ by $(R_2)$, so
$\varphi'(f)\geq \varphi(f)+\frac{1}{2}+2\times\frac{5}{4}=0$. Otherwise, $d(y)\leq 7$ and $d(z)\geq 10$. The vertex $x$ gives
$\frac 4 5$ to $f$ by $(R_2)$ and $z$ gives at least $\frac 7 5$ by $(R_1)$.  Moreover, $y$ gives at least $\frac 54$ to $f$ by $(R_1), (R_2)$ and $(R_3)$. Thus $\varphi'(f)\geq\varphi(f)+2\times\frac{4}{5}+\frac{7}{5}=0$.

Suppose that $d(x)=5$. Then $x$ gives at least $\frac 12$ to $f$. If $d(y)\ge 8$, then $\varphi'(f)\geq\varphi(f)+\frac{1}{2}+2\times \frac{5}{4}=0$ by $(R_1)$. If $d(y)=7$, then $x$ gives
$\frac {4}{5}$ to $f$ by $(R_3)$, so
$\varphi'(f)\geq\varphi(f)+\frac{4}{5}+2\times \frac{8}{7}>0$. Otherwise, $d(y)=5$ or 6. Note that $x$ (or $y$) gives at least $\frac 5 6$ to $f$ by $(R_1)$ and $(R_3)$. We consider the degree of $z$.

{\bf Case 1.}  $d(y)\geq 9$

Then
$\varphi'(f)\geq\varphi(f)+\frac{4}{3}+2\times\frac{5}{6}=0$.

{\bf Case 2.}  $d(y)=8$.

In this case, $x$ gives $\frac 54$ to $f$. Thus
$\varphi'(f)\geq\varphi(f)+\frac{5}{6}+2\times\frac{5}{4}>0$.

{\bf Case 3.}  $d(y)\le 7$.

In this case, $x$ sends 1 to $f$ by $(R_{3.3})$, $x$ (or $y)$ sends
at least $1$ to $f$ by $(R_1)$ and $(R_{3.3})$. Thus
$\varphi'(f)\geq\varphi(f)+3\times 1=0$.

In any case, we have $\varphi'(x)\geq 0$ for any $x\in V(G)\cup
F(G)$, a contradiction to (2). This completes the proof of Lemma
2.1.
\end{proof}



\section{Properties of acyclically edge $k$-critical graphs}

Let $\phi: E(G)\longrightarrow[k]$ be an edge $k$-coloring of $G$.
For a vertex $v\in V(G)$ and an edge $e=uv$, we say
that the color $\phi(e)$ $appears$ on $v$. Let
$C_{\phi}(v)=\{\phi(uv)|u\in N(v)\}$ and $F_{\phi}(uv)=\{\phi(vv')| v'\in N(v), v'\neq v\}$. Note that $F_{\phi}(uv)$ need not be the same as $F_{\phi}(vu)$.

Recall that a $multiset$ is a generalized set where a member can
appear multiple times. If an element $x$ appears $t$ times in a
multiset $\mathcal{S}$, then we say that the $multiplicity$ of $x$ in
$\mathcal{S}$ is $t$. In notation $mult_{\mathcal{S}}(x)=t$. The
cardinality of a finite multiset $\mathcal{S}$, denoted by
$||\mathcal{S}||$, is defined as $\sum_{x\in
\mathcal{S}}mult_{\mathcal{S}}(x)$. Let $\mathcal{S}_1$ and
$\mathcal{S}_2$ be two multisets. The $join$ of $\mathcal{S}_1$ and
$\mathcal{S}_2$, denoted as $\mathcal{S}_1\uplus \mathcal{S}_2$, is
the multiset that have all the members of $\mathcal{S}_1$ as well as
$\mathcal{S}_2$ such that for any $x\in \mathcal{S}_1\uplus
\mathcal{S}_2$, $mult_{\mathcal{S}_1\uplus
\mathcal{S}_2}(x)=mult_{\mathcal{S}_1}(x)+mult_{\mathcal{S}_2}(x)$.

Let $\alpha,\beta$ be two colors. An $(\alpha,\beta)$-$maximal$
$bichromatic$ $path$ with respect to $\phi$ is maximal path
consisting of edges that are colored $\alpha$ and $\beta$
alternatingly.  An $(\alpha,\beta,u,v)$-$critical$ $path$ is an $(\alpha,\beta)$-maximal bichromatic path which starts
at the vertex $u$ with an edge colored $\alpha$ and ends at $v$ with an edge colored $\alpha$.

A graph $G$ is called an $acyclically$ $edge$ $k$-$critical$ $graph$
if $\chi'_a(G)>k$ and any proper subgraph of $G$ is acyclically edge
$k$-colorable. Obviously, if $G$ is an acyclically edge $k$-critical
graph with $k>\Delta(G)$, then $\Delta(G)\ge 3$. The following facts
are obvious.

\begin{fact} \label{one-bichromatic-path}
Given a pair of color $\alpha$ and $\beta$ of a proper edge coloring
$\phi$ of $G$, there is at most one  $(\alpha,\beta)$-maximal
bichromatic path containing a particular vertex $v$, with respect to
$\phi$.
\end{fact}

\begin{fact} \label{uv-edge-degree}
Let $G$ be an acyclically edge $k$-critical graph, and $uv$ be an
edge of $G$. Then for any acyclically edge $k$-coloring $\phi$ of
$G-uv$, if $C_{\phi}(u)\cap C_{\phi}(v)=\emptyset$, then $d(u)+d(v)\ge
k+2$. If $|C_{\phi}(u)\cap C_{\phi}(v)|=t$, say $\phi(uu_i)=\phi(vv_i)$
for $i=1,2,..,t$, then $\sum\limits_{i=1}^{t}d(v_i)+d(u)+d(v)\ge
k+t+2$ and $\sum\limits_{i=1}^{t}d(u_i)+d(u)+d(v)\ge k+t+2$.
\end{fact}

In \cite{hou2011}, Hou et al.  considered the properties of
acyclically edge $k$-critical graphs and got the following lemmas.

\begin{lemma} \cite{hou2011}\label{2-connected}
Any acyclically edge $k$-critical graph  is $2$-connected.
\end{lemma}

\begin{lemma} \cite{hou2011}\label{lem3vertex}
Let $G$ be an acyclically edge $k$-critical graph with $k\ge
\Delta(G)+2$ and $v$ be a $3$-vertex of $G$. Then the neighbors of
$v$ are $(k-\Delta(G)+2)^+$-vertices.
\end{lemma}

In this section, we give some properties of acyclically edge
critical graphs, which will be used in the proof of Theorem
\ref{main-result}.

\begin{lemma}  \label{neighbor of 2-vertex}
Let $v$ be a $2$-vertex of   an acyclically edge $k$-critical graph
$G$ with $k>\Delta(G)$. Then the neighbors of $v$ are
$(k-\Delta(G)+3)^+$-vertices.
\end{lemma}

\begin{proof} Suppose to the contrary that $v$ has a neighbor $u$
whose degree is at most $k-\Delta(G)+2$. Let $N(v)=\{u,w\}$ and
$N(u)=\{v, u_1,u_2,\ldots,u_t\}$, where $t\le k-\Delta(G)+1$. Then
the graph $G'=G-uv$ admits an acyclic edge $k$-coloring $\phi$ by
the choice of $G$ with $\phi(uu_i)=i$ for $1\le i\le t$. Since
$d(u)+d(v)\le \Delta(G)+2$, we have $\phi(wv)\in C_{\phi}(u)$ by Fact
\ref{uv-edge-degree}, say $\phi(wv)=1$.   Then for any $t+1\le i\le
k$, there is a $(1,i,u,v)$-critical path
through $u_1$ and $w$ with respect to $\phi$, since otherwise we can color $uv$ with $i$
properly with avoiding bichromatic cycle, a contradiction
to the choice of $G$. Thus $\Delta(G)\geq d(u_1)\ge k-t+1$. This implies that
$t=k-\Delta(G)+1$ and $C_{\phi}(u_{1})=C_{\phi}(w)=\{1,k-\Delta(G)+2,\ldots,k\}$.
Recolor $wv$ with 2, by the same argument, there is a
$(2,i,u,v)$-critical path  through $u_2$ and $w$ for any $t+1\le
i\le k$. Now exchange the colors on $uu_1$ and $uu_2$, and color
$uv$ with $t+1$. The resulting coloring is an acyclic edge coloring
of $G$ using $k$ colors by Fact \ref{one-bichromatic-path}, a
contradiction.
\end{proof}

\begin{lemma}  \label{number of 2-vertex}
Let $v$ be a $t$-vertex of  an acyclically edge $k$-critical graph
$G$ with $k>\Delta(G)$ and $t\ge k-\Delta(G)+2$. Then $n_2(v)\le t+\Delta(G)-k-2$.
\end{lemma}

\begin{proof} If $n_2(v)=0$, we are done. Otherwise, $t\ge k-\Delta(G)+3$ by Lemma
\ref{neighbor of 2-vertex}.  By contradiction, suppose that $v$ has
neighbors $v_1,v_2,\cdots,v_t$ with $d(v_i)\ge 3$ for any $1\le i\le
m$, where $m\le k-\Delta(G)+1$, and $N(v_i)=\{v, u_i\}$ for any
$m+1\le i\le t$. The graph $G'=G-vv_t$ admits an acyclic edge
$k$-coloring $\phi$ by the choice of $G$ with $\phi(vv_i)=i$ for
$1\le i\le t-1$.  Since $d(v)+d(v_t)\le \Delta(G)+2$,
$\phi(u_tv_t)\in \{1,2,\cdots,m\}$ by Fact \ref{uv-edge-degree}, say
$\phi(u_tv_t)=1$. Then for any color $i$ with $t\le i\le k$, there
is a $(1,i,v,v_t)$-critical path through $v_1$ and $u_t$.  Note that for any $i$ with $m+1\le i\le t-1$, the
color $i$ should appear on $u_t$, since otherwise, we can recolor
$u_tv_t$ with $i$ and color $vv_t$ with a color appearing neither on
$v$ nor on $v_i$. Thus $\{m+1,\cdots,k\}\subseteq
C_{\phi}(u_t)$. Then $\Delta(G)\geq d(u_{t})\geq k-m+1$. This
implies that $m=k-\Delta(G)+1$ and $d(u_t)=\Delta(G)$. For any $1\le
i\le m$, and $t\le j\le k$,  there is a $(i,j,v,v_t)$-critical path through $v_i$ and
$u_t$ if we recolor $v_tv_t$ with $i$, since
otherwise, we can recolor $u_tv_t$ with $i$ and color $vv_t$ with
$j$.

{\bf Claim.} For any $m+1\le i\le t-1$, there is a $(1,i,v,v_t)$-critical path
through $v_1$ and $u_t$ with respect to $\phi$.

\begin{proof}
Otherwise, suppose that there exists a color $i_0$ with $m+1\le i_0\le
t-1$, such that there is not a $(1,i_0,v,v_t)$-critical path through $v_1$ and $u_t$. Color $vv_t$ with $i_0$, uncolor $vv_{i_0}$, the
colors on the other edges unchange,  we get an acyclic edge coloring
$\varphi$ of $G-vv_{i_0}$. It follows from Fact \ref{uv-edge-degree} that $\varphi(u_{i_0}v_{i_0})\in \{1,2,\cdots,m\}$.   Coloring
$vv_{i_{0}}$ with $t$ results in an acyclic edge coloring of $G$, a contradiction. \end{proof}

If we recolor $u_tv_t$ with 2, as in the proof above, there is a $(2,i,v,v_t)$-critical path
through $v_2$ and $u_t$ for any $m+1\le i\le t-1$. Thus $d(v_i)=\Delta(G)$ and $F_{\phi}(vv_i)=\{m+1, \cdots,k\}$ for $i=1, 2$. We exchange the colors on $vv_1$ and $vv_2$, and color $vv_t$ with $t$. The resulting coloring is an acyclic edge coloring of $G$ by Fact \ref{one-bichromatic-path}, a contradiction.
\end{proof}

\begin{lemma} \label{uv-edge-colored}
Let $uv$ be an edge of an acyclically edge $k$-critical graph $G$
with $k>\Delta(G)$. If $d(u)+d(v)\leq k-\Delta(G)+4$, then for any
acyclic edge $k$-coloring $\phi$ of $G-uv$, $|C_{\phi}(v)\cap
C_{\phi}(u)|\geq2$.
\end{lemma}

\begin{proof}
By contradiction, suppose that there exists an acyclic edge
$k$-coloring $\phi$ of $G-uv$ such that $|C_{\phi}(v)\cap
C_{\phi}(u)|\leq 1$. It follows from Fact \ref{uv-edge-degree} that
$|C_{\phi}(v)\cap C_{\phi}(u)|=1$, say $\phi(uu_1)=\phi(vv_1)=1$. Let $T=[k]\setminus {C_{\phi}(v)\cup C_{\phi}(u)}$.
Then $|T|=k-|C_{\phi}(v)\cup C_{\phi}(u)|\ge
k-(k-\Delta(G)+1)=\Delta(G)-1$. Moreover, for any color $i\in T$, there is
a $(1,i,u,v)$-critical path through $u_1$ and $v_1$. This implies
that $|T|=\Delta(G)-1$, $F_{\phi}(uu_1)=F_{\phi}(vv_1)=T$, and
$d(u)+d(v)=k-\Delta(G)+4\ge 5$. We may assume that
$d(u)\ge 3$. Let $u_2$ be another neighbor of $u$ different from
$u_1$ and assume that $\phi(uu_2)=2$. Then $2\notin C_{\phi}(v)$.
Recoloring $vv_1$ with 2 results in another acyclic edge
$k$-coloring $\phi'$ of $G-uv$. Similarly, there is a
$(2,i,u,v)$-critical path through $u_2$ and $v_1$ with respect to
$\phi'$ for any $i\in T$. Thus $F_{\phi}(uu_2)=T$. Then we exchange
the colors on $uu_1$ and $uu_2$, and color $uv$ with a color in $T$.
By Fact \ref{one-bichromatic-path}, the resulting coloring is an
acyclic edge coloring of $G$, a contradiction.
\end{proof}

\begin{lemma} \label{uv-edge-colored-2}
Let $v$ be $3^+$-vertex of an acyclically edge $k$-critical graph
$G$ with neighbors $u, v_1,v_2,\cdots,v_t$, where  $k>\Delta(G)$,
and $\phi$ be an acyclic edge $k$-coloring of $G-uv$. If
$d(u)+d(v)\leq k-\Delta(G)+4$ and $|C_{\phi}(v)\cap C_{\phi}(u)|=2$,
then the multiset $\mathcal{S}=F_{\phi}(vv_1)\uplus
F_{\phi}(vv_2)\uplus \cdots F_{\phi}(vv_t)$ contains at least $d(u)$
colors from $C_{\phi}(u)\cup C_{\phi}(v)$.
\end{lemma}

\begin{proof}
Assume that $\phi(vv_i)=i$ for $i=1,2,\cdots,t$, $\{1,2\}\subseteq C_{\phi}(u)$. Let $T=[k]\setminus {C_{\phi}(v)\cup C_{\phi}(u)}$.  Then $|T|=k-|C_{\phi}(v)\cup C_{\phi}(u)|\ge k-(k-\Delta(G))=\Delta(G)$, and for any color $i\in T$, there is either a $(1,i,u,v)$-critical path or a $(2,i,u,v)$-critical path.
Note that for $i=1,2$, $F_{\phi}(vv_i)\cap \{3,\cdots,t\}\neq \emptyset,$ since otherwise, we can choose a color from $T\setminus F_{\phi}(vv_i)$ to recolor $vv_i$ and get a new acyclic edge coloring $\phi'$ of $G-uv$ with $|C_{\phi'}(v)\cap C_{\phi'}(u)|=1$ by Fact \ref{one-bichromatic-path}, a contradiction to Lemma \ref{uv-edge-colored}.

Now we show  $\mathcal{S}$ contains at least $d(u)-2$ colors from
$C_{\phi}(u)$, which implies that $\mathcal{S}$ contains at least
$d(u)$ colors from $C_{\phi}(u)\cup C_{\phi}(v)$. Otherwise, let
$\alpha,\beta$ be two colors in $C_{\phi}(u)$ but  not contain in
$\mathcal{S}$. Then recolor $vv_1$ with $\alpha$ and $vv_2$ with $\beta$. The
resulting coloring is an acyclic edge coloring $\phi'$ of $G-uv$.
Then for any color $i\in T$, there is either an
$(\alpha,i,u,v)$-critical path through $v_1$, or an
$(\beta,i,u,v)$-critical path through  $v_2$ with respect to
$\phi'$. Assume that $\phi(uu')=\alpha$ and $\phi(uu'')=\beta$.
Since $|T|\ge \Delta(G)$, there exists a color $\gamma\in T\setminus
F_{\phi}(uu')$. Then $\gamma\in F_{\phi}(uu'')$ and there is a
$(\beta,\gamma,u,v)$-critical path through $u''$ and $v_2$. Now we
give an acyclic edge $k$-coloring $\varphi$ of $G$ by
$\varphi(vv_1)=\beta$, $\varphi(vv_2)=\alpha$, $\varphi(uv)=\gamma$,
and $\varphi(e)=\phi(e)$ for the other edges of $G$, a
contradiction.

\end{proof}



\section{Proof of Theorem \ref{main-result}}

 \begin{pf1}
By contradiction, let $G$ be a planar graph with a fixed embedding in the plane and suppose that $G$ is an acyclically edge $(\Delta(G)+10)$-critical graph, i.e., $G$ is the minimum counterexample to the theorem in terms of the number of edges. Then $G$ is 2-connected by Lemma \ref{2-connected}. Let $v$ be a vertex of $G$. If $d(v)=2$, then the neighbors of $v$ are $13^+$-vertices by Lemma \ref{neighbor of 2-vertex}. If $d(v)=3$, then the neighbors of $v$ are $12^+$-vertices by Lemma \ref{lem3vertex}. If $d(v)\ge 12$, then $n_2(v)\le d(v)-12$ by Lemma \ref{number of 2-vertex}.

Now we show that $G$ contains the configuration $(\mathcal{A}_3)$ or $(\mathcal{A}_4)$. Let $H$ be the graph obtained by deleting all 2-vertices from $G$. Then for any vertex $v\in V(H)$, either $d_H(v)=d_G(v)\ge
3$, or $d_H(v)\ge 12$ by Lemma \ref{number of 2-vertex}. It follows from Lemma \ref{unavoidable
configurations} that there exists a $5^-$-vertex $v$ in $H$  such that $v$ belongs to one of the configurations $(\mathcal{A}_2), (\mathcal{A}_3), (\mathcal{A}_4)$. This implies that $d_{G}(v)=d_{H}(v)$ and $v$ belongs to one of the
configurations $(\mathcal{A}_2), (\mathcal{A}_3), (\mathcal{A}_4)$ in $G$ by Lemma \ref{number of 2-vertex}. However, if $d(v)=3$, then the neighbors of $v$ are $12^+$-vertices by Lemma \ref{lem3vertex}. Thus $G$ contain the configuration $(\mathcal{A}_3)$ or $(\mathcal{A}_4)$.

Suppose that $v$ has neighbors $v_1,v_2,\cdots,v_t$, where $t=3,4$,
with $d(v_1)\le d(v_2)\le \cdots \leq d(v_t)$, $v_1$ has neighbors
$v, w_1,w_2,\cdots,w_s$, and let $k=\Delta(G)+10$ for simplicity.
Then the graph $G'=G-vv_{1}$ admits an acyclic edge $k$-coloring
$\phi$ by the choice of $G$  with $\phi(v_1w_i)=i$ for $1\le i\le s$. Since $d(v)+d(v_1)\le 11$,
$|C_{\phi}(v)\cap C_{\phi}(v_1)|\geq2$ by Lemma
\ref{uv-edge-colored}. Let $T=[k]\setminus (C_{\phi}(v)\cup
C_{\phi}(v_1))$ and $\mathcal{S}=F_{\phi}(vv_2)\uplus \cdots
\uplus F_{\phi}(vv_t)$. Then  Now we derive an acyclic edge $k$-coloring of $G$
from $\phi$, a contradiction.

\medskip

 \textbf{Case 1.} $d(v)=4$ with $d(v_{1})\leq 7$ and $d(v_{2})\leq 9$.

\medskip

Note that $|T|=k-(s+1)\ge \Delta(G)+3$ and $||\mathcal{S}||\le 2(\Delta(G)-1)+8= 2\Delta(G)+6$. We consider the number of colors in $C_{\phi}(v)\cap C_{\phi}( v_{1})$.

\medskip

{\bf Subcase 1.1. } $|C_{\phi}(v)\cap C_{\phi}( v_{1})|=2.$

\medskip

Suppose that
$\phi(vx_i)=\phi(v_{1}w_{i})=i$ for $i=1,2$, where $i\in \{v_2,v_3,v_4\}$. Then for any color $s+2\le i\le
k$, there is either a $(1,i,v,v_1)$-critical path through $x_1$ and $w_1$, or a
$(2,i,v,v_1)$-critical path through $x_2$ and $w_2$ with respect to $\phi$. Since
$\mathcal{S}$ contains at least  $d(v_1)$ colors from
$C_{\phi}(v)\cup C_{\phi}(v_1)$ by Lemma \ref{uv-edge-colored-2},
there exists a color $\alpha\in T$ such that
$mult_{\mathcal{S}}(\alpha)=1$, say $\alpha\notin F_{\phi}(vv_{2})$.
Recoloring $vv_{2}$ with $\alpha$ results in an acyclic edge
coloring $\phi'$ of $G-vv_{1}$ with $|C_{\phi'}(v)\cap
C_{\phi'}(v_1)|=1$, a contradiction to Lemma \ref{uv-edge-colored}.

\medskip

 {\bf Subcase 1.2}
$|C_{\phi}(v)\cap C_{\phi}(v_{1})|=3.$

\medskip

In this case, $|T|= k-s\ge \Delta(G)+4$. This implies that there exists a
color $\alpha\in T$ such that $mult_{\mathcal{S}}(\alpha)=1$, say
$\alpha\notin F_{\phi}(vv_2)$.  Recoloring $vv_2$ with $\alpha$
results in an acyclic edge coloring $\phi'$ of $G-vv_{1}$ with
$|C_{\phi'}(v)\cap C_{\phi'}(v_1)|=2$, and the argument in Subcase
1.1 works.

\medskip

 {\bf Case 2}  $d(v)=5,$ with $d(v_{1})\leq6, d(v_{2})\leq7$ and
 $d(v_{3})\leq8$.

\medskip

Note that $|T|=k-(s+2)= \Delta(G)+8-s$ and $||\mathcal{S}||\le 2(\Delta(G)-1)+6+7= 2\Delta(G)+11$. We consider the number of colors in $C_{\phi}(v)\cap C_{\phi}( v_{1})$.

{\bf Subcase 2.1} $|C_{\phi}(v)\cap C_{\phi}(v_{1})|=2.$

Suppose that
$\phi(vx_i)=\phi(v_{1}w_{i})=i$ for $i=1,2$, where $i\in \{v_2,v_3,v_4\}$. Then for any color $s+2\le i\le
k$, there is either a $(1,i,v,v_1)$-critical path through $x_1$ and $w_1$, or a
$(2,i,v,v_1)$-critical path through $x_2$ and $w_2$ with respect to $\phi$. Since $\mathcal{S}$ contains at
least $s+1$ colors from $C_{\phi}(v)\cup C_{\phi}(v_1)$ by Lemma
\ref{uv-edge-colored-2}, there exists a color $\alpha\in T$, such that
$mult_{\mathcal{S}}(\alpha)=1$, say $\alpha\notin F_{\phi}(vv')$.
Recoloring $vv'$ with $\alpha$ results in an acyclic edge coloring
$\phi'$ of $G-vv_{1}$ with $|C_{\phi'}(v)\cap C_{\phi'}(v_1)|=1$, a
contradiction to Lemma \ref{uv-edge-colored}.

\medskip

{\bf Subcase 2.2} $|C_{\phi}(v)\cap C_{\phi}(v_{1})|=3.$

\medskip

Suppose that $N(v)=\{v_1,x_1,\cdots,x_4\}$, $\phi(vx_i)=\phi(v_{1}w_{i})=i$
for $i=1,2,3$ and $\phi(vx_4)=s+1$. Then for any color $i\in T$,
there is either a $(1,i,v,v_1)$-critical path through $x_1$, or a
$(2,i,v,v_1)$-critical path through $x_2$, or a
$(3,i,v,v_1)$-critical path through $x_3$. Let
$\mathcal{S'}=F_{\phi}(vx_{1})\uplus F_{\phi}(vx_{2})\uplus
F_{\phi}(vx_{3}) $. Note that $|T|=k-(s+1)=\Delta(G)+9-s\geq
\Delta(G)+4 $, $||\mathcal{S'}||\le 2(\Delta(G)-1)+7=2\Delta(G)+5$.
There exists a color $\alpha\in T$, such that
$mult_{\mathcal{S'}}(\alpha)=1$, say $\alpha\in F_{\phi}(vx_{1})$.
Recoloring $vx_{2}$ with $\alpha$ results in a proper edge coloring
$\phi'$ of $G-vv_{1}$ with $|C_{\phi'}(v)\cap C_{\phi'}(v_1)|=2$. If
there is no bichromatic cycle with respect to $\phi'$, we are done. Otherwise, there is a
$(s+1,\alpha,v,x_{2})$-critical path through $x_{4}$ with respect to
$\phi'$. Let $\phi''(vx_3)=\alpha$, $\phi''(e)=\phi(e)$ for the
other edges $e\in G-vv_1$. Then $\phi''$ is an acyclic edge coloring
of $G-vv_1$ with $|C_{\phi''}(v)\cap C_{\phi''}(v_1)|=2$ by Fact
\ref{one-bichromatic-path} and the argument in Subcase 2.1 works.

\medskip

\medskip

{\bf Subcase 2.3}  $|C_{\phi}(v)\cap C_{\phi}(v_{1})|=4$.

\medskip

Suppose that $\phi(vv_i)=i-1$ for $2\le i\le 5$. Recolor $vv_2$ with a color
from $T\setminus F_{\phi}(vv_2)$ and get an proper edge coloring
$\phi'$ of $G-vv_1$. If  $\phi'$ does not contain bichromatic cycle,
then $\phi'$ is acyclic with  $|C_{\phi'}(v)\cap
C_{\phi'}(v_{1})|=3$, and the argument in Subcase 2.2 works.
Otherwise, $F_{\phi}(vv_2)\cap\{2,3,4\}\neq \emptyset$. Similarly,
for $i=3,4,5$, $F_{\phi}(vv_i)\cap\{1,2,3,4\}\neq \emptyset$. Thus
$\mathcal{S}$ contains at least four colors from $\{1,2,3,4\}$. Since
$|T|=k-s\ge \Delta(G)+5$ and $||\mathcal{S}||\le 2\Delta(G)+11$,
there exists a color $\alpha\in T$, such that $mult_{\mathcal{S}}(\alpha)=1$,
say $\alpha\notin F_{\phi}(vv_1)$. Recoloring $vv_1$ with $\alpha$
results in an acyclic edge coloring $\phi''$ of $G-vv_{1}$ with
$|C_{\phi''}(v)\cap C_{\phi''}(v_1)|=3$, the argument in Subcase
2.2 works.

In any case, we can get an acyclic edge coloring of $G$ with
$\Delta(G)+10$ colors, a contradiction.
\end{pf1}

Determining the acyclic chromatic index of a graph is a hard
problem both from theoretical and algorithmic points of view. Even
for complete graphs, the acyclic chromatic index is still not
determined exactly.

\begin{problem} Determine  the acyclic  chromatic index of
planar graphs.
\end{problem}


\begin{thebibliography}{99}
\bibitem{Alon1991} {N. Alon},
A parallel algorithmic version of the Local Lemma,  Random Structures
Algorithms 2 (1991) 367--378.
\bibitem{Alon2001} N. Alon, B. Sudakov and  A. Zaks,  Acyclic edge
colorings of graphs,  J. Graph Theory 37 (2001) 157--167.


\bibitem{Alon2002} N. Alon,  A. Zaks, Algorithmic aspects of acyclic
edge colorings, Algorithmica  32 (2002) 611--614.



\bibitem{and2012} L.D. Andersen,E. M\'{a}\v{c}ajov\'{a}, J. Maz\'{a}k, Optimal acyclic
edge-coloring of cubic graphs, J. Graph Theory DOI 10.1002/jgt.20650.

\bibitem{basaSIAM} M. Basavaraju, L.S. Chandran, N. Cohen, F. Havet and T. M\"{u}ller,
Acyclic edge-coloring of planar graphs,  SIAM J. Discrete Math 25
(2) (2011) 463--478.
\bibitem{basa2009} M. Basavaraju, L.S. Chandran, Acyclic edge coloring of graphs with maximum degree
4,  J. Graph Theory 61 (3) (2009) 192--209.
\bibitem{basa2008} M. Basavaraju, L.S. Chandran, Acyclic edge coloring of
subcubic graphs,  Discrete Mathematics 308 (2008) 6650--6653.





\bibitem{boro2009} M. Borowiecki, A. Fiedorowicz, Acyclic edge colouring of
planar graphs without short cycles,  Discrete Mathematics 310 (2010)
1445--1455.


\bibitem{fie2008} A. Fiedorowicz,  M. Ha{\l }uszczak, N. Narayanan, About acyclic edge colourings of planar graphs,  Information Processing Letters 108 (2008) 412--417.



\bibitem{hou2011}J. Hou, N. Roussel, J. Wu, Acyclic chromatic index of planar graphs with triangles,  Information Processing Letters, 111 (2011) 836--840.

\bibitem{hou2009} J. Hou, J. Wu, G. Liu and B. Liu, Acyclic edge chromatic
number of outerplanar graphs,  J. Graph Theory, 64 (2010) 22--36.
\bibitem{hou2008} J. Hou, J. Wu, G. Liu and B. Liu, Acyclic edge colorings
of planar graphs and seriell-parallel graphs,  Sciences in China A 52
(2009) 605--616.


\bibitem{molloy1998} {M. Molloy,  B. Reed}, Further Algorithmic
Aspects of the Local Lemma,  Proceedings of the 30th Annual ACM
Symposium on Theory of Computing (1998) 524--529.






\end{thebibliography}
\end{document}